\documentclass[12pt]{amsart}
\usepackage{amsthm,amssymb}

\newtheorem{theorem}{Theorem}[section]

\newtheorem{lemma}[theorem]{Lemma}

\newtheorem{corollary}[theorem]{Corollary}
\theoremstyle{remark}
\newtheorem{remark}[theorem]{Remark}	

\newtheorem{conjecture}[theorem]{Conjecture}
\theoremstyle{definition}

\newcommand\Z{\mathbb{Z}}
\newcommand\Q{\mathbb{Q}}
\newcommand\R{\mathbb{R}}
\newcommand\C{\mathbb{C}}
\renewcommand\P{\mathbb{P}}

\newcommand\alg{\mathrm{a}}
\newcommand\T{\mathbb{T}}
\newcommand\eps{\varepsilon}
\renewcommand\O{\mathcal{O}}
\newcommand\pid{\mathfrak{p}}
\newcommand\aid{\mathfrak{a}}
\newcommand\ld{\mathfrak{d}}
\newcommand\arch{\mathrm{arch}}

\newcommand\disc{\operatorname{disc}}

\newcommand\Fermat{\mathcal{F}}

\begin{document}
\title{Two Restricted ABC Conjectures}
\author{Machiel van Frankenhuijsen}
\address{Department of Mathematics,
Utah Valley University,
Orem,
Utah 84058-5999,
vanframa@uvu.edu}


\begin{abstract}
Ellenberg proved that the abc conjecture would
follow if this conjecture were known for sums $a+b=c$ such that $D\mid abc$ for some integer~$D$.
Mochizuki proved a theorem with an opposite restriction,
that the full abc conjecture would follow if it were known for abc sums that are not highly divisible.
We prove both theorems for general number fields.
\end{abstract}

\maketitle

\section{Introduction}

The abc conjecture,
formulated in 1985 by Masser and Oesterl\'e,
states that for all $\eps>0$,
there exists a constant $M$ such that
\begin{gather}\label{E:abc}
\log c<\sum_{p\mid abc}\log p+\eps\log c+M
\end{gather}
for all sums $a+b=c$ of positive coprime integers (briefly,
 an ``abc sum'').
Here,
the {\em radical\/}
of the abc sum involves
 all prime divisors $p$ of $abc$,
\[
r(a,b,c)=\sum_{p\mid abc}\log p,
\]
and the {\em height\/} of the abc sum is
\[
h(a,b,c)=\log c.
\]
These quantities will be defined for general number fields in Section~\ref{S:hr}.

In 1999,
Ellenberg showed that it would suffice to prove this conjecture for sums such that $abc$ is divisible by a fixed integer.
His result~\cite{E} is that
if $D$ is a positive integer and $M$ is a constant such that~\eqref{E:abc} holds
for all abc sums such that~$D\mid abc$,
then
for all abc sums,
\begin{gather*}
h(a,b,c)<r(a,b,c)+n\eps h(a,b,c)+M+\log(2n2^n),
\end{gather*}
where $n$ is the Euler indicator
$n=D\prod\nolimits_{p\mid D}\bigl(1-1/p\bigr)$.

Ellenberg proved his result for integers,
but his argument applies equally well to number fields.
For number fields,
$D$ is naturally an arithmetic divisor
(defined in Section~\ref{S:E}),
because the condition $D\mid abc$ is trivially satisfied for all primes with~\mbox{$p\nmid D$}.

We explain several formulations of the abc conjecture in Section~\ref{S:abc},
and formulate and prove the generalization of Ellenberg's theorem in Section~\ref{S:E}.
The idea of the proof is to replace an abc sum
$(x:1-x:1)$ by $(x^n:1-x^n:1)$,
related to the cover $x\mapsto x^n$ from $\P^1$ to $\P^1$.

\begin{remark}
In~\cite{E},
the abc sum is replaced by $A+B=C$,
where $A=(a-b)^n$,
$B=c^n-(a-b)^n$,
and $C=c^n$.
This is related to the cover $x\mapsto(2x-1)^n$ of $\P^1$.
With our choice of cover,
we have a slightly stronger estimate at the prime $2$.
\end{remark}

When we try to include the archimedean valuations in the conditions imposed by $D$,
we encounter the problem that power maps are expansive near $1$,
making it hard to simultaneously satisfy conditions at different archimedean places.
If the conjugates of $a$ lie on the unit circle,
then methods from Diophantine approximation yield an exponent so that the conjugates lie close to $1$.
On the other hand,
conjugates may lie at a distance of $e^{-h(a,b,c)}$ from the unit circle,
and then an exponent of size $e^{h(a,b,c)}$ is needed to make the conjugates diverge to $0$ or $\infty$.
With an exponent depending on the height,
we have been unable to guarantee that the error is of order $o(h)$.
For this reason,
Theorem~\ref{T:E} is formulated with only one archimedean condition.

\begin{remark}
The absolute value of an algebraic number is real algebraic.
By Roth's theorem,
if an algebraic number does not lie on the unit circle,
then it must lie at a certain minimal distance.
This may make it possible to estimate a minimal rate at which powers must diverge from the unit circle.
\end{remark}

In 2009,
Mochizuki proved that the abc conjecture follows from a weaker conjecture in the opposite direction.
To introduce Mochizuki's result,
we state it for sums of coprime rational integers.
We write $v_p(n)$ for the exponent of $p$ in the prime factorization of $n$.

Let $V$ be a finite set of such valuations.
Assume that for every even\footnote{See the text after~\eqref{E:G_v} for an explanation why $G$ must be even if $v_2\in V$.} integer~$G$,
there exists a function~\mbox{$\psi_G$} such that $\psi_G(h)=o(h)$ and
\begin{gather}\label{E:res abc}
h(a,b,c)<r(a,b,c)+\psi_G(h(a,b,c))
\end{gather}
for every abc sum such that for every~$v_p$ in $V$,
$v_p(abc)\leq v_p(G)$.
Mochizuki's result~\cite{M} is that then,
 for every $\eps>0$,
there exists a real number $M$ such that~\eqref{E:abc} holds for all abc sums.

The condition $v_p(abc)\leq v_p(G)$ means that~$a/c$ is $p$-adically not very close to the {\em tripod\/} of three points on the projective line (see Section~\ref{S: tripod}),
\begin{gather}\label{E:tripod}
\T=\{0,1,\infty\}.
\end{gather}
Since $a/c$ will be close to $\T$ for some primes,
this condition can only be imposed for the finitely many valuations in $V$.
In other words,~$G$
is naturally a neighborhood of $\T$ in a finite product of completions of~$\Q$.

The idea of the proof is to lift an abc sum to a point on the Fermat curve,
\begin{gather}\label{E:Fermat}
\Fermat\colon x^n+y^n=1,
\end{gather}
for a value of the exponent that is large enough,
depending on $\eps$,
\begin{gather}\label{E:n}
n\geq3+\frac6\eps,
\end{gather}
using the map $(x,y)\mapsto x^n$ from $\Fermat$ to $\P^1$.
This is a cover of the projective line by a hyperbolic curve of degree $n^2$.
Thus already to deduce the abc conjecture for $\Q$,
we need to assume~\eqref{E:res abc} for certain algebraic number fields of degree at most $n^2$ over $\Q$.
We refer to Theorem~\ref{T:M} for the complete formulation.

\begin{remark}
The value for~$n$ in Theorem~\ref{T:E}
depends on $D$.
On the other hand,
in Section~\ref{S:M},~$n$ is chosen to satisfy~\eqref{E:n} and does not depend on $G$.
\end{remark}

The paper~\cite{M} proves that the restricted abc conjecture implies Voj\-ta's Height Inequality (VHI,
see~\cite{V}),
 and leaves the abc conjecture as the special case
of VHI applied to~$\T$.
Since the abc conjecture implies VHI effectively (see~\cite{abcvhi,vF}),
we prove here only that Mochizuki's restricted abc conjecture implies the abc conjecture.
This simplification allows us to fix the choice of height functions and we only need~\cite[Proposition~1.7(i)]{M} for the Fermat curve.
Another technical simplification is in the use of a $p$-adic valuation to separate noncritical points from ramified points of a Belyi map in Section~\ref{S:Belyi},
where~\cite{Belyi} uses the archimedean valuation (i.e.,
the ordering of $\Q$ as real numbers).
We systematically write the error term as a function of the height,
which better clarifies our choice of $\eps$.
Apart from these changes in the presentation,
we have not extended or improved upon~\cite[Theorem~2.1]{M}.

\subsection{Acknowledgements}

In~\cite{E},
Ellenberg mentions that his result was known to the experts,
 and he cites~\cite{O}.
We call it Ellenberg's theorem,
but it should probably be attributed to Szpiro and Oesterl\'e.

We thank Mochizuki for many insightful discussions,
and the Research Institute for Mathematical Sciences of the University of Kyoto for their hospitality from January until August of 2018.

\section{The height and the radical}
\label{S:hr}

In a number field $K$,
a rational prime number $p$ factors as a product of prime ideals of the ring $\O_K$ of integers of $K$,
\begin{gather}\label{E:e}
p\O_K=\pid_1^{e_1}\cdots\pid_g^{e_g}.
\end{gather}
Each prime ideal defines an extension $\O_K/\pid_i$ of the finite field $\Z/p\Z$ of a degree
denoted by $f_i$.
The cardinality of this finite field is called the {\em norm\/} of the prime ideal,
\[
N\pid_i=p^{f_i}=\#(\O_K/\pid_i).
\]

A prime ideal defines a {\em$\pid$-adic valuation\/} $v_\pid$,
defined as the exponent of $\pid$ in the prime factorization of an ideal,
$
\aid=\prod_\pid\pid^{v_\pid(\aid)}$.
Conversely,
a valuation determines its prime ideal
\mbox{$
\pid_v=\{x\in\O_K\colon v(x)>0\}
$}.

The prime factors $\pid_i$ of $p$ correspond to embeddings of $K$ in~$\C_p$.
There are also embeddings in the field of complex numbers,
and each embedding $\imath\colon K\rightarrow\C$ gives an {\em archimedean valuation,\/}
\[
v_{\imath}(x)=-\log|\imath x|.
\]
An archimedean valuation does not define a prime ideal,
but after~\eqref{E:sum} below,
we formally define its norm.

If the embeddings $\imath$ and $\jmath$ are related by the continuous automorphism of~$\C$ (complex conjugation),
then the corresponding valuations are the same on $K$.
Likewise,
two embeddings of $K$ into $\C_p$ define the same valuation (and the same prime ideal of $\O_K$) if and only if they are related by a continuous automorphism of $\C_p$.
There are~$e_\pid f_\pid$ automorphisms of $\C_p$ (and embeddings of $K$) corresponding to the $\pid$-adic valuation.

We use the subscript $v$ for nonarchimedean valuations.
On the other hand,
the subscript $\imath$ stands for an embedding in $\C$,
and a pair of complex embeddings $\imath,\bar\imath$ corresponds to one archimedean valuation of $K$.
To count only valuations,
we could use the notation $v_{\imath,\bar\imath}(x)=-2\log|\imath x|$ to refer to the complex archimedean valuations (where $2=e_\imath f_\imath$;
also see~\cite[III.3, page~222]{Neukirch}).
We have not done this in this paper,
and have instead kept the inconsistency that archimedean valuations are counted by the embeddings of $K$ into $\C$.

The sum of all $p$-adic valuations and archimedean valuations is the trivial valuation:
for all nonzero $x\in K$,
\begin{gather}\label{E:sum}
\sum_\imath v_\imath(x)+\sum_vv(x)\log N\pid_v=0.
\end{gather}
We interpret this to mean that $\log N\pid_\imath=1$ if $\imath K$ is a subfield of $\R$,
and~$\log N\pid_{\imath,\bar\imath}=2$ if~$\imath K$ is not a real subfield.\medskip

For a point $(a,b,c)$ over $K$,
 the {\em local height\/} for each valuation is
\[
h_v(a,b,c)=-(\log N\pid_v)\min v\{a,b,c\},
\]
where $v\{a,b,c\}$ is the set of valuations of $a$,
$b$ and $c$.
For the archimedean valuations,
we use this formula as well,
with the norm~\mbox{$\log N\pid_\imath=1$} for each embedding in $\C$.
Each archimedean contribution can be written as
\begin{gather}\label{E:h_i}
h_\imath(a,b,c)=\log\max\{|\imath a|, |\imath b|, |\imath c|\}.
\end{gather}

The {\em global height\/} is defined as
\[
h_K(a:b:c)=\sum_{\imath}h_\imath(a,b,c)+\sum_vh_v(a,b,c).
\]

It follows from~\eqref{E:sum} that the global height only depends on $(a,b,c)$ up to multiples:
$h_K(a:b:c)=h_K(xa:xb:xc)$ for any nonzero~\mbox{$x\in K$}.
Thus~$h_K$ depends on the projective point,
as the notation indicates.

\begin{remark}\label{R:FS metric}
We could define the archimedean contributions to the height more naturally as
\[
h^\circ_\imath(a,b,c)=\log\sqrt{|\imath a|^2+ |\imath b|^2+ |\imath c|^2},
\]
using the spherical distance on the Riemann sphere $\P^1(\C)$.
This differs from our choice~\eqref{E:h_i} by a bounded function.
\end{remark}

For each nonarchimedean valuation,
the {\em local radical\/} of $(a,b,c)$ is
\[
r_v(a:b:c)=\begin{cases}
\log N\pid_v	&\text{ if }\ \#v\{a,b,c\}\geq2,\\
0		&\text{ if }\ v(a)=v(b)=v(c).
\end{cases}
\]
Clearly,~$r_v$ depends only on the projective point.
For each embedding of $K$ in $\C$,
the archimedean contribution to the radical is defined as
\[
r_\imath(a:b:c)=1,
\]
because $\#\{|\imath a|,|\imath b|,|\imath c|\}\geq2$ always (with only one exception,
see~\eqref{E: rho sum} below).
The {\em global radical\/} is
\begin{gather}\label{E:r_K}
r_K(a:b:c)=[K:\Q]+\sum_vr_v(a:b:c).
\end{gather}

The radical counts the valuations where the orders of $a$,
$b$ and $c$ are not all equal,
 with a weight $\log N\pid_v$.
The degree $[K:\Q]$ in~\eqref{E:r_K} is the number of embeddings of~$K$ in $\C$,
and each embedding contributes $1$.\medskip

The following two lemmas are needed for our generalization of Ellenberg's theorem.

\begin{lemma}\label{L: powers height}
For\/ $a\in K$ and a positive integer\/ $n,$
\[
h_K(a^{n-1}+a^{n-2}+\dots+1:1)\leq [K:\Q]\log n+(n-1)h_K(a:1).
\]
\end{lemma}

\begin{proof}
Clearly,
$h_v(a^{n-1}+\dots+1,1)=(n-1)h_v(a,1)$ for every nonarchimedean valuation.
For archimedean valuations,
\[
\max\{|\imath a^{n-1}+\dots+1|,1\}\leq\begin{cases}
n|\imath a^{n-1}|	&\text{ if }|\imath a|>1,\\
 n				&\text{ if }|\imath a|\leq1.
 \end{cases}
\]
It follows that
$h_\imath(a^{n-1}+\dots+1,1)\leq\log n+(n-1)h_\imath(a,1)$.
Adding all contributions,
the lemma follows.
\end{proof}

This lemma and the next have an even nicer formulation using the absolute height (i.e.,
when divided by $[K:\Q]$),
to be introduced in Section~\ref{S:absolute} below.
The content of both lemmas becomes particularly clear when we view the height as the degree of a cover
(see Section~\ref{S: tripod} below and Remark~\ref{R:FF}).

\begin{lemma}\label{L:n height}
For every algebraic number\/ $a\in K,$
\[
h_K(a^n:1-a^n:1)\leq nh_K(a:1-a:1)+[K:\Q]\log2,
\]
and
\[
h_K(a^n:1-a^n:1)\geq nh_K(a:1-a:1)-[K:\Q]\cdot n\log2.
\]
\end{lemma}

\begin{proof}
Clearly,
$h_v(a^n,1-a^n,1)=nh_v(a,1-a,1)$ for each nonarchimedean valuation.
For the archimedean valuations we use the inequality
$
\max\{|a|,1\}\leq\max\{|a|,|1-a|,1\}\leq 2\max\{|a|,1\}
$
for a complex number\/ $a,$
to estimate
$
h_\imath(a^n,1-a^n,1)\geq nh_\imath(a,1-a,1)-n\log2,
$
and similarly in the opposite direction.
\end{proof}

\subsection{The tripod zero, one and infinity}
\label{S: tripod}

It is not always possible to cancel a common factor from a sum $a+b=c$ over $K$ to obtain a point with relatively prime integral coordinates.
Another point of view is obtained when one divides by $c$ to obtain a sum of algebraic numbers
\begin{gather}\label{E:a+b=1}
a+b=1.
\end{gather}

The height can be computed as $h_K(a:b:1)$.
Each nonarchimedean local height $h_v(a,b,1)=0$ unless $v(a)<0$,
in which case we say that~``$a$ has a pole at $v$''.
Such valuations are counted in the radical,
along with the valuations where~$v(a)>0$ (``$a$ has a zero at $v$'') or $v(1-a)>0$ (``$a$ takes the value $1$ at $v$'').
With this point of view,
the height of a sum~\eqref{E:a+b=1} is the degree of the pole divisor of $a$.

In the function field case,
the radical is the number of points in the preimage $a^{-1}\T$,
where $\T$ is the divisor of the three points $0$,
$1$,~$\infty$ as in~\eqref{E:tripod}.
For number fields,
`points' are valuations,
 counted with a weight of $\log N\pid_v$.
The archimedean valuations are always counted (with weight $1$ if $\imath K\subset\R$ or $2$ if~$\imath K\not\subset\R$) because usually~$|\imath a|$ or~$|\imath b|\ne1$.
The only exception is the sum of sixth roots of unity in~$K=\Q[\sqrt{-3}\,]$,
where we have defined the radical of the sum
\begin{gather}\label{E: rho sum}
e^{\pi i/3}+e^{-\pi i/3}=1
\end{gather}
as
$
r_K(e^{\pi i/3}:e^{-\pi i/3}:1)=2$,
but maybe it should be defined as $0$ ?\medskip

The following lemma is needed in the proof of Mochizuki's theorem.

\begin{lemma}\label{L:r<h}
For\/ $a\in K$ and not in\/ $\T,$
\[
r_K(a:1-a:1)\leq3[K:\Q]+ 3h_K(a:1-a:1).
\]
\end{lemma}

\begin{proof}
For any valuation with $v(a)>0$,
the contribution to the radical is bounded by the local height,
\[
r_v(a:1-a:1)\leq(\log N\pid_v)\max\{v(a),0\}=h_v(1/a,1-1/a,1).
\]
And for an embedding in $\C$ we simply add $1$ to get an inequality,
\begin{gather}\label{E:r_i}
1=r_\imath(a:1-a:1)\leq1+h_\imath(1/a,1-1/a,1).
\end{gather}
It follows that the contribution of the valuations where $a$ has a zero is bounded by $[K:\Q]+h_K(1/a:1-1/a:1)$.
The second term equals $h_K(a:1-a:1)$ since $(1/a:1-1/a:1)=(1:1-a:a)$.
The contributions where $v(1-a)>0$ or $v(a)<0$ are bounded similarly.
\end{proof}

\subsection{The absolute height and radical}
\label{S:absolute}

For an extension $E$ of $K$,
the global height
$h_E$ equals $[E:K]\, h_K$ on points with coordinates in~$K$.
The {\em absolute height,\/}
\begin{gather}\label{E:h}
h(a:b:c)=\frac1{[F: \Q]}h_F(a:b:c),
\end{gather}
is defined for every algebraic point and does not
depend on the choice of a field of definition $F$ of the point.\medskip

There is no absolute radical that could be defined independently of the choice of a field of definition.
The reason for this becomes clear in the function field case:
an extension of function fields corresponds to a cover of algebraic curves,
and even if a function lies in the smaller function field,
 on the covering curve it is a new function.
More concretely,
in the function field $\C(x)[y]$ with $y^2=4x^3-g_2x-g_3$,
$x^2$ is not a polynomial but an elliptic function,
with its own branch points.

On algebraic points,
the {\em absolute radical\/} is defined using the minimal field of definition,
\begin{gather}\label{E:abs rad}
r(a:b:c)=\frac1{[F:\Q]}r_F(a:b:c),
\end{gather}
where $F=\Q[a,b,c]$ and coordinates are chosen such that $a$,
$b$ or~$c$ equals~$1$.
For an extension of $F$,
\[
r_K(a:b:c)\leq[K:\Q]r(a:b:c),
\]
with equality if $K/F$ is unramified at each valuation that contributes to the radical.\medskip

Another quantity that depends on the chosen field of definition is the discriminant of the field containing the coordinates of a point.
The {\em logarithmic different\/} is defined by the formula
\begin{gather}\label{E:ld}
\ld(a:b:c)=\frac1{[F:\Q]}\log\disc(F),
\end{gather}
where $\disc F$ is the discriminant of the minimal field of definition of the point $(a:b:c)$.
For an extension,
$\log\disc(K)\geq [K:\Q]\ld(a:b:c)$,
with equality if and only if the extension $K/F$ is unramified.

\begin{remark}\label{R:FF}
For a function field $K$ over a field of constants $k$ of characteristic zero,
the height of an abc sum $a+b=1$ is the degree of~$K$ over the field of rational functions~$k(a)$,
and the radical is the number of points in the preimage $a^{-1}\T$ on the curve defined by $K$.
The analogue of the abc conjecture,
\[
h_K\leq 2g-2+r_K,
\]
where $g$ is the genus of~$K$,
 is a simple consequence of Hurwitz' formula.
The analogue of $2g-2$ for number fields is~$\log\disc(K)$.
\end{remark}

\subsection{The height and radical on a curve}

We only need the height of a point on the Fermat curve~\eqref{E:Fermat},
and this allows us to fix the choice of a height function.

In general,
for a cover $\phi\colon C\rightarrow\P^1$ and a positive divisor $D$ on~$\P^1$,
\begin{gather}\label{E:h phi}
h_{\phi^*D}(x)=h_D(\phi x),
\end{gather}
and we simply define the height with respect to $D$ on $\P^1$ by
\[
h_D(x)=(\deg D)h(x:1-x:1).
\]

The radical with respect to a positive divisor~$D$ on a curve~$C$ is
\[
r_D(x)=\frac1{[F:\Q]}\sum_{v(f(x))>0}\log N\pid_v
\]
where $F$ is the field of definition of the point $x$ and the sum is over all valuations of $F$ where $x$ is $v$-adically close to $D$.
This means that one of the defining equations $f=0$ for the Cartier divisor $D$ has a positive valuation at~$x$.
For example,
the absolute radical defined in~\eqref{E:abs rad} above is~$r_{\T}(a/c)$.
The choice of the equations~$f=0$ is not unique,
but the radical depends on this choice by a bounded function.
It is possible that~$\log N\pid_v$ occurs with a multiplicity in the sum,
namely if~$v(f(x))>0$ for several $f$,
but this happens only for finitely many valuations and with a bounded multiplicity.
For the line $a+b=c$ in~$\P^2$,
if we choose the equations $a/c=0$,
$a/c=1$,
and~$c/a=0$ for the three points of $\T$,
each contribution to the radical occurs once.

Clearly,
$r_D=r_{D'}$,
where $D'$ is the reduced divisor of $D$ (the points of $D$ with multiplicity $1$).
The generalization of Lemma~\ref{L:r<h} is
\begin{gather}\label{E:r D'}
r_D(x)\leq \deg D'+h_{D'}( x).
\end{gather}
The term $\deg D'$ comes from the archimedean valuations.

\section{The abc conjecture}
\label{S:abc}

We write $h_K=h_K(a:b:c)$ and $r_K=r_K(a:b:c)$ for the height and the radical of an abc sum $a+b=c$ in the number field $K$.

\begin{conjecture}\label{C:M}
Let\/ $K$ be a number field and $\eps>0$.
There exists an~$M$ such that for all sums $a+b=c$ in $K,$
\[
h_K\leq r_K+\eps h_K+M.
\]
\end{conjecture}

It is known that $M$ must be unbounded as $\eps\to0$.
If $M$ is effectively known as a function of $\eps$ then
\begin{gather}\label{E:psi from h}
\psi(h)=\inf_{\eps>0}\eps h+M(\eps)
\end{gather}
is the {\em error term\/} in the abc conjecture.
A function~$\psi$ thus obtained satisfies
$
\psi(h)=o(h),
$
i.e.,
$\psi(h)/h\to0$ as $h\to\infty$,
but in the following conjecture,
we also allow $\psi(h)=\eps h+M$.
For instance,
for $\eps=1/2$ and $K=\Q$,
it is likely that~\mbox{$M(1/2)=0$} is possible,
because every abc sum of rational numbers so far has satisfied~$h_\Q\leq 2r_\Q$.

\begin{conjecture}[Effective abc]\label{C:psi}
Let\/ $K$ be a number field.
There exists a function~$\psi$ such that for all sums $a+b=c$ in $K,$
\[
h_K\leq r_K+[K:\Q]\psi(h),
\]
where $
h=h_K/[K:\Q]$
is the absolute height of $(a:b:c)$,
as in~\eqref{E:h}.
\end{conjecture}

In an extension $E/K$ and a point defined over $K$,
$r_E\leq [E:K]r_K$,
with equality only if $E/K$ is unramified at each valuation that contributes to the radical.
Thus the function $\psi$ in Conjecture~\ref{C:psi} must depend on~$K$.
On the other hand,
according to~\cite[(3.3)]{abcvhi},
\[
\log\disc(E)+r_E(a:b:c)\geq [E:K]\bigl(\log\disc(K)+r_K(a:b:c)\bigr).
\]
With the addition of the discriminant,~$\psi$ could be independent of $K$.
The strongest inequality is obtained when~$K$ is the minimal field of definition of an abc sum.

\begin{conjecture}[Uniform abc]\label{C:disc}
There exists a function~$\psi$ such that for all sums $a+b=c$ of algebraic numbers,
\[
h_K\leq \log\disc(K)+r_K+[K:\Q]\psi(h),
\]
where $K$ is any field of definition of the point $(a:b:c)$.

With the absolute radical~\eqref{E:abs rad} and logarithmic different~\eqref{E:ld},
this conjecture can be formulated as
\[
h\leq \ld+r+\psi(h)
\]
for all algebraic points $(a:b:c)$ on the line $a+b=c$.
\end{conjecture}

Numerical evidence suggests that $M$ in Conjecture~\ref{C:M} may grow like~$4/\eps$.
By~\eqref{E:psi from h},
$\psi(h)$ could then be $4\sqrt{h}$ as in the next conjecture.

\begin{conjecture}[Falsifiable abc]\label{C:sqrt}
For all sums $a+b=c$ of algebraic numbers,
$
h\leq\ld+r+4\sqrt{h}.
$
\end{conjecture}

This would be close to the strongest possible form of the abc conjecture,
because there are infinitely many abc sums of rational numbers such that
\begin{gather}\label{E:good abc}
h_\Q\geq r_\Q+6\sqrt{h}/\log h
\end{gather}
(see~\cite{goodabc}).
Before Mochizuki's proof of the abc conjecture,
it was only known that~$\psi$ could be~$h-C\log h$ for some constant
(see~\cite{ST,SY}).

\begin{remark}
As was pointed out by Sarnak in a conversation with the author,
an important weakness of Conjecture~\ref{C:M} is that it cannot be falsified.
The function $4\sqrt h$ in Conjecture~\ref{C:sqrt} is a completely explicit bound for $h-\ld-r$,
and by~\eqref{E:good abc},
is close to best possible.
This function makes Conjecture~\ref{C:sqrt} falsifiable,
and we start to see a glimpse of an archimedean contribution to the radical.
It may have a deep geometric significance,
as Remark~\ref{R:Nevanlinna} below explains.
Conjecture~\ref{C:sqrt} may be formulated as
\[
\sqrt{h}\leq2+\sqrt{\ld+r+4},
\]
and the author will award \$1,000 for the first counter example to this inequality,
even with $h_\imath$ replaced by $h^\circ_\imath$ of Remark~\ref{R:FS metric}.
\end{remark}

\begin{remark}\label{R:Nevanlinna}
For number fields,
we have defined the archimedean contribution to the absolute radical as $1$,
independent of the abc sum.
For the field of meromorphic functions,
a naive definition of the radical defines the archimedean contribution as $1$ as well
(if we had defined~$\deg(v_z,\rho)=1$ in~\cite{meromorphic}).
In that paper,
we defined another archimedean contribution~$r_\arch$ to the radical to obtain $h\leq r+r_\arch$ for all sums $a+b=1$ of meromorphic functions,
by a direct and easy application of the definitions.
The usefulness of this theorem (for example,
to prove Fermat's Last Theorem for holomorphic functions) comes from the estimate $r_\arch\leq (2+\eps)\log h$,
which follows from the Lemma of the Logarithmic Derivative of Nevanlinna Theory.
We expect that for number fields,
the term $4\sqrt{h}$ in Conjecture~\ref{C:sqrt} is similarly a bound for a
natural archimedean contribution to the radical,
as yet unknown.
\end{remark}

\section{Ellenberg's theorem for number fields}
\label{S:E}

An arithmetic divisor is a finite sum
$
D=\sum_\imath D_\imath(\imath)+\sum_vD_v(v),
$
where each $D_v\in\Z$ and $D_\imath\in\R$ for each embedding of $K$ in $\C$.
The divisor is {\em positive\/} if all nonzero coefficients are positive.

The condition in Ellenberg's theorem is conveniently formulated using a positive arithmetic divisor,
but we have only been able to include one archimedean condition.

We write $K_v$ for the completion of the number field $K$ at the valuation~$v$,
and $p$ for the rational prime number such that~\mbox{$v(p)>0$}.
Recall from~\eqref{E:e} that valuations are integer valued,~\mbox{$v(p)=e$} and~\mbox{$N\pid_v=p^f$}.

An element $a\in K_v$ is an integer if $v(a)\geq0$,
and then $v(a-1)\geq0$ as well.
The identity
\begin{gather}\label{E:a^m}
a^m-1=(a-1)(a^{m-1}+\dots+1)
\end{gather}
shows that $v(a^m-1)\geq v(a-1)$ if $v(a)\geq0$.
In other words,
power maps are not expansive near $1$.
By the following lemma,
the $p$-th power function is a contraction near $1$.

\begin{lemma}\label{L:nonarch}
Let\/ $a\in K_v$ be such that\/ $v(a-1)\geq1$ and let\/ $p$ be the residue characteristic.
Then
$
v(a^p-1)\geq1+v(a-1).
$
\end{lemma}

\begin{proof}
Identity~\eqref{E:a^m} applied to $m=0,\dots,p-1$ gives that each power of $a$ is close to $1$.
It follows that
\[
a^p-1=(a-1)\sum_{m=0}^{p-1}\bigl(1+O(a-1)\bigr)=(a-1)(p+O(a-1)).
\]
Since $v(p)\geq 1$,
we conclude that $v(a^p-1)\geq 1+v(a-1)$.
\end{proof}

The next two lemmas explain that the power map on complex numbers close to the unit circle is expansive near $1$.
The only attractors are $0$ and~$\infty$ in~$\P^1(\C)$,
and $1$ is a repelling fixed point.

\begin{lemma}\label{L:argument}
Let\/ $\alpha$ be a complex number with\/ $|\alpha|=1$ and\/ $d$ a positive real number.
Then there exists an exponent\/ $1\leq m\leq 8\pi e^d$ such that
\[
|\alpha^m-1|<\frac{1}{4}e^{-d}.
\]
\end{lemma}

\begin{proof}
Let $k$ be the integer part of $8\pi e^d$.
Order $\alpha^0$,
$\alpha^1$,
\dots,
$\alpha^k$ according to increasing principal argument along the unit circle.
Since there are at least~$8\pi e^d$ points and the circumference is $2\pi$,
there are two points that are at most $e^{-d}/4$ apart.
Their quotient is a power $\alpha^m$ equal to~$e^{ it}$ for some positive real number $t<e^{-d}/4$.
Since $|e^x|=1$ on the line from $0$ to $it$,
the distance to~$1$ can be estimated as
\[
|e^{ it}-1|=\biggl|\int_0^{ it}e^x\,dx\biggr|<\frac{1}{4}e^{-d}.
\]
We conclude that $|\alpha^m-1|<e^{-d}/4$ for some $m\leq8\pi e^d$.
\end{proof}

\begin{lemma}\label{L:arch}
Let\/ $a$ be a complex number and let\/ $d$ be a positive real number such that\/ $\bigl|\log|a|\bigr|<e^{-2d}/32\pi$.
Then there exists an exponent\/ $1\leq m\leq8\pi e^d$ such that\/ $|a^m-1|<e^{-d}$.
\end{lemma}

\begin{proof}
First consider the argument $\alpha=a/|a|$ of $a$.
By Lemma~\ref{L:argument},
there exists an $m\leq8\pi e^d$ such that $|\alpha^m-1|<e^{-d}/4$.
For this exponent,
$|m\log|a||<e^{-d}/4$,
hence $|a|^m<1/(1-e^{-d}/4)$.
Since
$
a^m-1$ can be written as $|a|^m(\alpha^m-1)+|a|^m-1$,
it follows that
\[
|a^m-1|<\frac{e^{-d}}{4}\left(1+\frac{1+e^{-d}/4}{1-e^{-d}/4}\right)\leq\frac23e^{-d},
\]
and this is bounded by $e^{-d}$.
\end{proof}

To formulate the condition of Ellenberg's theorem,
we use that a number $A\in K$ is $v$-adically close to $\T=\{0,1,\infty\}$ if for some $D_v$,
\begin{gather}\label{E:D-v}
v(A)\geq D_v\ \text{ or }\ v(1/A)\geq D_v\ \text{ or }\ v(1-A)\geq D_v.
\end{gather}
For an embedding $\imath\colon K\rightarrow\C$,
the analogous condition is
\begin{gather}\label{E:D-i}
-\log|\imath A|\geq D_\imath\ \text{ or }\ \log|\imath A|\geq D_\imath\ \text{ or }\ -\log|1-\imath A|\geq D_\imath .
\end{gather}

\begin{remark}
If $A+B=C$ is an abc sum as in the introduction and $D=\prod_pp^{D_v}$ is the integer associated with the arithmetic divisor,
then condition~\eqref{E:D-v} means that $D\mid ABC$.
Condition~\eqref{E:D-i} means that~\mbox{$C\geq e^{D_\imath}$}.
\end{remark}

\begin{theorem}\label{T:E}
Let\/ $K$ be a number field and\/ $\imath\colon K\rightarrow\C$ an embedding.
Let\/ $D$ be a positive arithmetic divisor such that\/ $D_\jmath=0$ for every other embedding in $\C$.
Suppose that an increasing function $\psi_D$ is known such that
for every sum\/ $A+B=1$ in\/ $K$ satisfying\/~\eqref{E:D-v} for all nonarchimedean valuations and\/~\eqref{E:D-i}
for the archimedean valuation\/~$\imath$ of\/~$K,$
\begin{gather}\label{E:D-abc}
h_K\leq\log\disc(K)+r_K+[K:\Q]\psi_D(h).
\end{gather}
Then for every sum $a+b=1$ in\/ $K,$
\[
h_K\leq\log\disc(K)+r_K+[K:\Q]\psi_D(Ch+\log2)+\log\bigl( C2^C\bigr),
\]
where\/ $C$ is given in~\eqref{E:bound n} and~\eqref{E: n_0} below.
\end{theorem}

\begin{proof}
For $a\in K$
we will find an exponent such that  the sum $A+B=1$ with $A=a^n$ satisfies the conditions~\eqref{E:D-v} and~\eqref{E:D-i}.
Which exponent needs to be chosen depends on the closeness of $|\imath a|$ to $1$ and it is the reason that no second archimedean condition can be imposed with our approach.
The number of possible exponents,
and their size,
is bounded in terms of $D$.

If $v(a)=0$ for a valuation with $D_v\geq1$ then $v(a^{N\pid_v-1}-1)>0$ (since the group $(\O_K/\pid_v)^*$ has order $N\pid_v-1$).
By Lemma~\ref{L:nonarch},
if~$n$ is a multiple of $(N\pid_v-1)p_v^{D_v-1}$ then $v(1-A)\geq D_v$.
And if~$v(a)\ne0$ then~$|v(A)|\geq D_v$ for~$n\geq D_v$.
Since $p_v^{D_v-1}\geq2^{D_v-1}\geq D_v$,
both conditions are satisfied for every nonarchimedean valuation if $n$ is a multiple of
\begin{gather}\label{E: n_0}
n_0=\prod_{v\colon D_v\geq1}\bigl(p_v^{f_v}-1\bigr)p_v^{D_v-1}.
\end{gather}
If $D_\imath=0$ then we use this exponent for $n$.

If $D_\imath>0$,
then $n$ will be a multiple of $n_0$ as follows.
If
\[
\bigl|n_0\log|\imath a|\bigr|\leq\frac1{32\pi}e^{-2D_\imath},
\]
then we can find a power~\mbox{$m\leq 8\pi e^{D_\imath}$}
such that $|1-\imath A|\leq e^{-D_\imath}$ for the exponent~\mbox{$n=n_0m$},
by Lemma~\ref{L:arch}.
And if $\bigl|n_0\log|\imath a|\bigr|>\frac1{32\pi}e^{-2D_\imath}$,
then for the least integer $m$ larger than~$D_\imath\cdot 32\pi e^{2D_\imath}$,
the exponent $n=n_0m$ satisfies $\bigl|\log|\imath A|\bigr|\geq D_\imath$.
The exponent $n$ thus obtained is bounded by
\begin{gather}\label{E:bound n}
C=\bigl(32\pi D_\imath e^{2D_\imath}+1\bigr)n_0.
\end{gather}

In summary,
if $D_\imath=0$ for all archimedean valuations,
then we use~$A=a^n$ with $n=n_0$.
If $D_\imath>0$ and $\imath a$ lies close to the unit circle,
then we use $n=mn_0$ for some $m\leq8\pi e^{D_\imath}$.
And if $\imath a$ does not lie very close to the unit circle,
we use an exponent close to~\eqref{E:bound n}.
This gives a finite number of possible exponents to consider.
We  conclude that for one sum $A+B=1$ thus obtained,
\eqref{E:D-abc} holds,
\begin{gather}\label{E:abc hypothesis}
H_K\leq \log\disc(K)+R_K+[K\colon\Q]\psi_D(H),
\end{gather}
where $H_K=h_K(A:B:1)$,
$H$ is the corresponding absolute height,
and~$R_K$ is the radical of $(A:B:1)$.
We write $h$
and $r$ for the corresponding quantities for $(a:1-a:1)$,
using the absolute height and radical for convenience.

By Lemma~\ref{L:n height} and~\eqref{E:abc hypothesis},
the absolute height is estimated by
\[
nh\leq H+n\log2\leq\ld(a:1-a:1)+R+\psi_D(H)+n\log2.
\]

To estimate the radical,
if $v(A)\ne0$ then $v(a)\ne0$ and $R_v=r_v$.
And if $v(A-1)>0$ then $A-1=(a-1)S$,
where $S=a^{n-1}+\dots+1$,
 and either $v(a-1)>0$,
which means that $v$ contributes to the radical $r$,
or~$v(S)>0$.
We find that
\[
R\leq r+\sum_{v(S)>0}\log N\pid_v\leq r+\sum_v(\log N\pid_v)\max\{v(S),0\}.
\]
Adding archimedean contributions $\sum_\imath \max\{-\log|\imath S|,0\}$ to the right-hand side,
we obtain by Lemma~\ref{L: powers height},
\[
R\leq r+h(1:1/S)\leq r+(n-1)h(a:1)+\log n.
\]

Clearly,
$h(a:1)\leq h$.
Further,
since $H\leq nh+\log2$ by Lemma~\ref{L:n height} and~$\psi_D$ is increasing,
we find
\[
h\leq \ld+r+\psi_D(nh+\log 2)+\log \bigl(n2^n\bigr).
\]
The theorem follows since $n\leq C$ and $\psi_D$ is increasing.
\end{proof}

\section{Belyi maps}
\label{S:Belyi}

The following is Mochizuki's extension~\cite{Belyi} of Belyi's theorem,
that every curve defined over a number field can be mapped to $\P^1$ by a mapping that is only ramified above the tripod $\T$,
and does not map any chosen finite subset of rational points of the curve to $\T$.
Equivalently,
the map is regular at the points of any chosen finite subset.

The first step reduces this to the curve $\P^1$ and one point where the map $\P^1\rightarrow\P^1$ is required to be regular.

Let $C$ be a curve of genus $g$ defined over a number field $K$ and let $E$ and $R$ be finite subsets defined over $K$ such that none of the conjugates of points in~$R$ lies in $E$.
Enlarge $R$ if necessary so that the associated positive divisor (of points of multiplicity one) has degree at least $2g$.
By Riemann-Roch,
for each point $r$ in $R$,
$l(R)=l(R-(r))+1$.
It follows that there exists a function with a simple pole at $r$,
at most a simple pole at each other point of~$R$,
and no other poles.
By considering the divisor of all conjugates of $r$,
we can even find such a function that is defined over $K$.

\begin{lemma}
Let\/ $C$ be a curve over a number field\/ $K$
with a finite set of points\/ $R=\{r_1,r_2,\dots,r_n\}$.
For each\/ $r_k,$
let\/ $f_k$ be
a function defined over\/~$K$
with a simple pole at\/ $r_k,$
at most a simple pole at the other points,
and no other poles.
Then there exists a function
defined over\/~$K$ and
with pole divisor\/ $R=(r_1)+\dots+(r_n).$
\end{lemma}

\begin{proof}
The residues of each $f_k$ at each of its poles form a finite set of algebraic numbers,
and we can find a prime number~$p$ such that each is a $p$-adic unit (in one embedding in $\C_p$).
Consider the function
\[
f=\sum_{k=1}^np^kf_{k}.
\]
Clearly,
$f$ has only simple poles,
and only at the points of $R$.
The residue of $f$ at any of these points is a sum of at least one nonzero term none of whose $p$-adic valuations coincide.
Hence the residue is nonzero.
\end{proof}

By the lemma,
we find a function
\[
f\colon C\longrightarrow\P^1
\]
which maps each point of $R$ to $\infty$ (which we now label $r$),
and which is unramified at each of these points.
Moreover,
the image $f(E)$ is a finite set of algebraic numbers defined over $K$.
We add to this set the branch locus~$\{f(x)\colon f\text{ is ramified at }x\}$.
We thus obtain a set,
again denoted $E$,
of algebraic numbers
such that~$r=\infty\not\in E$
and if $f$ is ramified at $x\in C$ then $f(x)\in E$.

Next we find a rational function such that $E$ is mapped to $\T$,
the function is only ramified above $\T$,
and $r$ is not mapped to $\T$.

Let $p$ be a prime number such that each element of $E$ is a $p$-adic integer and the degree as an algebraic number of each element of $E$ is less than $p$.
For example,
if $E$ is empty,
then we take $p=2$.
The function
\begin{gather}\label{E:to 1/p}
z\longmapsto \frac{z}{1+pz}
\end{gather}
maps $\infty$ to $1/p$ and every element of $E$ to an algebraic number of degree less than $p$ which is also a $p$-adic integer.
If $E$ contains only rational points then we continue with the folding map below.
Otherwise,
consider one algebraic element $\alpha$ of $E$ of maximal degree $d$ and with monic minimum polynomial $m$.
Since $v_p(\alpha)\geq0$,
the coefficients of $m$ are rational numbers without a factor $p$ in the denominator.
The function
\begin{gather}\label{E:m}
z\longmapsto m(z)
\end{gather}
maps each element of $E$ to an algebraic number of at most the same degree,
and it maps $\alpha$ to $0$.
It is ramified at $\infty$,
which is mapped to~$\infty$.
By the following lemma,
applied with $M=0$ and $v_p(d)=0$ (by the choice of $p$),
each finite number $x$ where $m$ is ramified,
i.e.,
a root of the derivative $m'$,
is a $p$-adic integer.

\begin{lemma}\label{L:nona}
Let\/
$
m(X)=X^d+m_1X^{d-1}+\dots+m_{d-1}X+m_d
$
be a polynomial of degree\/ $d$ such that its roots satisfy\/ $v(\alpha_i)\geq M$ for a nonarchimedean valuation\/ $v$.
Then\/ $v(x)\geq M-v(d)$ for every root of\/~$m'(X)=dX^{d-1}+(d-1)m_1X^{d-2}+\dots+m_{d-1}$.
\end{lemma}

\begin{proof}
The root satisfies~\mbox{$-dx^{d-1}=(d-1)m_1x^{d-2}+\dots+m_{d-1}$}.
Since $m_k$ is a sum of products of $k$ roots,~$v(m_k)\geq kM$,
 and we find
\[
v(d)+(d-1)v(x)\geq\min_{1\leq k\leq d-1}kM+(d-1-k)v(x).
\]
If~$v(x)\leq M$,
then the minimum is attained at $k=1$ and the resulting inequality is equivalent to $v(x)\geq M-v(d)$.
And if on the other hand~$v(x)> M$,
then~$v(x)\geq M-v(d)$ as well,
since $v(d)\geq0$.
\end{proof}

The polynomial $m$ of~\eqref{E:m} maps any number with $v_p(x)<0$ to such a number again.
Since roots of $m'$ are algebraic numbers of degree at most~$d-1$,
replacing $E$ by $m(E\cup\infty)\cup\{m(x)\colon m'(x)=0\}$,
after finitely many steps,
there are no more numbers in $E$ of degree $d$.
Repeating this with the numbers of degree $d-1$ in $E$,
 after finitely many steps,
each element of $E$ is a rational number whose denominator is not divisible by~$p$.
The set $R$ where the composition is required to be regular is mapped to a rational number~$r$ whose denominator is divisible by $p$ and is thus not mapped to an element of $E$.
In other words,
the composition so far is a map~$C\rightarrow\P^1$ that is regular at each point of $R$.

Finally,
we reduce the number of rational points in $E$ by successively applying a folding map.
If $\infty$ is not an element of $E$ (i.e.,
if after step~\eqref{E:to 1/p},
$E$ only contains rational points),
then we adjoin it to $E$.

If $E$ contains at least three rational numbers $e_0<e<e_1$,
then we map~$e_0$ and $e_1$ to $0$ and $1$ by
\[
z\longmapsto\frac{z-e_0}{e_1-e_0}.
\]
This maps $e$ to a rational number between $0$ and $1$,
which we write as~$\frac{m}{m+n}$ for positive integers $m$ and $n$.
It maps $r$ and any~$e\in E$ to rational numbers such that $v_p(r)<v_p(e)$.
Then we fold $1$ onto $0$ by
\[
z\longmapsto z^m(1-z)^n.
\]
This polynomial has a local maximum at a point between $0$ and $1$,
which is easily computed to be $\frac{m}{m+n}$,
and it has no other branch points besides this point and $0$,
$1$,
and $\infty$.
Thus $E$ is mapped to a set containing one less rational number.
Also,
$r$ is mapped to a rational number with a $p$-adic valuation less than any point of $E$.
After finitely many applications of this folding map,
$E$ will only contain $0$,
$1$ and $\infty$.

The map $C\rightarrow\P^1$ is now obtained as the composition of the maps in the above steps.
Every point in $E$ is mapped to $\T$,
and $R$ is not mapped to $\T$.

\subsection{Finite sets of noncritical points}
\label{S:uniform Belyi}

There is a sequence of Belyi maps $f_0,f_1,\dots$ such that for any finite set $R$ disjoint from $E$,
one of the functions $f_k$ is not ramified at the points of $R$.
The construction so that $f_0,\dots,f_n$ suffices for sets of $n$ elements (i.e.,
$k\leq \#R$) is taken from~\cite{SZ}.

Let $f_0\colon C\rightarrow \P^1$ be any map such that $f_0(E)\subseteq\T$ and it is only ramified at points above $\T$.
Clearly,
$f_0$ solves the problem for~$n=0$ and the empty set $R_0$.
Suppose that for $n\geq0$,
$f_n$ is a Belyi map for $E$ and a set $R_n$ such that $f_n(r)\not\in\T$ for every point in $R_n$.
In particular,~$E\subseteq f_n^{-1}(\T)$.
Let then $f_{n+1}$ be a Belyi map for~$E$ and the set
\[
R_{n+1}=R_{n}\cup (f_{n}^{-1}(\T)-E).
\]

Since~$R_n\subseteq R_{n+1}$,
the sets $R_{n+1}-R_n$ are disjoint.
It follows that for a set $R$ of $n$ points not in $E$,
$R_{k+1}-R_k$ must be disjoint from $R$ for some~$k\leq n$.
For that value of $k$,
if $f_k(r)\in\T$ for an element of~$R$ then~$r\in R_{k+1}$ but not in~$R_k$,
but that contradicts the choice of $k$.
Hence $f_k$ is regular at each point of $R$.

\section{The Chevalley--Weil theorem}

Mochizuki proves Lemma~\ref{L:C P} and~\ref{L:C P e} below in~\cite[Proposition~1.7(i)]{M}.
They are reminiscent of the Chevalley--Weil theorem that the discriminant of the field of definition of a point above a rational point is bounded.
We briefly explain these lemmas here,
but we refer to Mochizuki's paper for the complete proofs,
especially for the estimates at wildly ramified primes.

\begin{remark}
Wild ramification is an arithmetic phenomenon in mixed characteristic,
and does not happen in function fields of characteristic zero.
The closest function field analogue is an inseparable extension of a function field over a finite field.
\end{remark}

In the following two lemmas,
$\phi\colon C\rightarrow\P^1$ is a Belyi map and $D=\phi^*\T$.
The first lemma is the first inequality in~\cite[Proposition 1.7(i)]{M}.

\begin{lemma}\label{L:C P}
There exists an\/ $M$ such that for every number field\/ $K$ and all points\/ $x\in C(K),$
\begin{gather*}\label{E:C P}
\ld(\phi x)+r_\T(\phi x)\leq\ld(x)+ r_D(x)+M.
\end{gather*}
\end{lemma}

Note that $\phi x$ may be defined over a subfield of $K$.
Mochizuki takes $K=\Q^\alg$,
 making it even clearer that $M$ does not depend on the field of definition of $x$.\medskip

The next lemma is slightly weaker than the second inequality of~\cite[Proposition 1.7(i)]{M}.
In Section~\ref{S:M},
we consider Belyi maps with ramification index $e$ at each point of $\T$,
and we refer to~\cite{M} for the stronger inequality involving the least common multiple of the ramification indices above $\T$.

\begin{lemma}\label{L:C P e}
There exists an\/ $M$ such that for every number field\/ $K$ and all points\/ $x\in C(K),$
\begin{gather*}\label{E:C P e}
\ld(x)\leq\ld(\phi x)+r_\T(\phi x)+M.
\end{gather*}
\end{lemma}

Just as in the previous lemma,
the mention of $K$ is inessential:
this is again a general statement about the algebraic points of $C$.
We refer to \cite[page~10]{M} for the proof.
The idea is that both sides are sums of local contributions
that are equal for unramified primes.
For tamely ramified primes,
both sides are equal for the primes that contribute to the radical,
and otherwise,
it is an inequality.
For wildly ramified primes,
we can assume that $K$ contains a $p$-th root of unity,
and restrict to the wild ramification group,
which is a composition of cyclic extensions of degree $p$.
For that part of the extension,
the exponent of the different can be bounded by $p$.

\section{Mochizuki's theorem}
\label{S:M}

Mochizuki notes that~\cite[Proposition 1.7(ii)]{M}
is an easy consequence of Hurwitz' formula.
We explain the essential contents here,
which is a computation of the degree of a map and of the reduced divisor~\eqref{E:D'} below.

By~\cite[formula~(4.2)]{abcvhi} (see also~\cite{vF}),
if $\phi$ is a Belyi function then
\begin{gather}\label{E:omega}
\omega=\phi^*(\infty)-\phi^{-1}\T,
\end{gather}
is a canonical divisor.
In particular,
$2g-2=\deg\phi-\#\phi^{-1}\T$.

Let $\phi\colon C\rightarrow\P^1$ be a Belyi function with ramification index\/ $e$ at each point above\/ $\T$.
Then the degree of $\phi$ is $ne$,
where $n$ is the number of points above each point of $\T$.
The genus of $C$ is $2g-2=n(e-3)$ and the pull-back of $\T$ is the divisor $D=e(\phi^{-1}\T)$ of degree $3ne$.
Its reduction is
\begin{gather}\label{E:D'}
D'=\phi^{-1}\T,
\end{gather}
a set of $3n$ elements.

An example is the Belyi function for the Fermat curve~\eqref{E:Fermat},
\begin{gather}\label{E:phi Fermat}
\phi\colon(x,y)\longmapsto x^n.
\end{gather}
The ramification of this map is of order $e=n$ at each point above $\T$,
and the degree is~$n^2$.
The genus of the Fermat curve is
$
2g-2=n(n-3).
$
In Lemma~\ref{L:D=0} below,
we need another Belyi map for $\Fermat$ with extra properties.\medskip

Given a finite set $V$ of valuations of $\Q$,
we consider compact neighborhoods of $\T$ in $\P^1(\C_v)$ for each $v\in V$.
We allow any kind of neighborhood such that the complement in $\P^1(K_v)$
for any number field $K$
 is a nonempty open set.
For nonarchimedean valuations,
such a neighborhood is
\begin{gather}\label{E:G_v}
G_v=\{x\in \C_v\colon v(x)\geq g_v\text{ or }v(x)\leq-g_v\text{ or }v(1-x)\geq g_v\}.
\end{gather}
If $g_v>1$,
 the complement of $G_v$ is a nonempty open subset of~$\P^1(K_v)$.
If the residue class field of $K_v$ has more than two elements (in particular,
for every $v\ne v_2$),
 we can even take~\mbox{$g_v>0$}.
For archimedean valuations,
we can use similar neighborhoods,
\[
G_\imath=\bigl\{x\in\C\colon |\imath x|\leq e^{-g_\imath}\text{ or }|\imath x|\geq e^{g_\imath}\text{ or }|1-\imath x|\leq e^{-g_\imath}\bigr\},
\]
for $g_\imath>0$,
as is easily verified by a drawing on the Riemann sphere or on the real line (if $K_v=\R$ and $g_\imath$ is a small positive number,
then $G_\imath$ leaves a small interval around $-1$ in the complement).

\begin{remark}\label{R:finite degree}
Every neighborhood of $\T$ in $\P^1(\R)$,
 $\P^1(\C)$ or $\P^1(K_v)$ contains a compact neighborhood as above.
Also,
$\P^1(\C)$ is locally compact,
but this is not the case for $\P^1(K)$ if $K$ is an infinite extension of $\Q_v$,
both if the residue class field is infinite or if the value group is not discrete.
But certain infinite extensions,
such as~$\Q(\sqrt2,\sqrt3,\sqrt5,\dots)$,
have a locally compact completion in every valuation.

Indeed,
for any valuation,
extensions of bounded degree of $\Q$ are contained in a finite extension of $\Q_v$,
hence in a locally compact field.
This fact follows from the Galois cohomology of local fields and is easy for quadratic fields:
every quadratic extension of $\Q$ is contained in the field~$\Q_2(i,\sqrt{2},\sqrt{3})$,
 and for $p\ne2$,
in $\Q_p(\sqrt{p},\sqrt{a})$,
where $a\in\Z$ is such that $x^2=a$ has no solutions modulo $p$.
\end{remark}

Suppose that for every number field\/ $K$ and for every system of neighborhoods~$G_v$ of~$\T$ as above a function $\psi_G(h)=o(h)$ is known such that
\begin{gather}\label{E:abc general position}
h(P)\leq\ld(P)+ r_\T(P)+\psi_G(h(P))
\end{gather}
for all $P=(x:1-x:1)$ in\/ $K$ such that $\sigma x\not\in G_v$ for every $v\in V$ and every embedding $\sigma\colon K\rightarrow\C_v$.

\begin{remark}
In~\eqref{E:abc general position} a function $\psi_G$ must be known for every system of neighborhoods of $\T$.
On the other hand,
the set of valuations $V$ is fixed and finite.
For Mochizuki's proof of the abc conjecture,
the only conditions are at the archimedean and the two-adic valuations.
\end{remark}

The following lemma is the {\em``claim''\/} of~\cite[page 12]{M}.
It proves that a weaker abc conjecture suffices to effectively find points on the Fermat curve (or on any curve of genus at least $2$).
Given $\eps$,
we use the Fermat curve with exponent~\eqref{E:n}.
We also assume that $\eps<1$,
so that $n\geq9$.
We use $h_\omega$,
the height of a point on $\Fermat$ with respect to a canonical divisor~\eqref{E:omega}.

\begin{lemma}\label{L:D=0}
Let\/ $K$ be a number field and assume~\eqref{E:abc general position}.
For all but finitely many points\/ $x$ on the Fermat curve over\/ $K,$
\begin{gather}\label{E:Mordell}
h_\omega(x)\leq\ld(x)+(\eps/2) h_\omega(x).
\end{gather}
\end{lemma}

\begin{proof}
If the inequality is false,
then there exists an infinite sequence of counter examples,
\begin{gather}\label{E:sequence}
h_\omega(x_n)>\ld(x_n)+(\eps/2) h_\omega(x_n).
\end{gather}
For a subsequence,
$\sigma x_n$ converges in $\C_v$,
say to $c_{v\circ\sigma}$,
for every $v\in V$ and every embedding of $K$ in $\C_v$.
Of these limit points,
some may be algebraic.
By the result of Section~\ref{S:Belyi},
we can find a Belyi map $\phi$ that maps the algebraic limit points not to $\T$.
The transcendental points are also not mapped to $\T$,
hence we can find compact neighborhoods~$G_v$ of~$\T$ such that $c_{v\circ\sigma}\not\in G_v$.
Then $\sigma x_n\not\in G_v$ for all but finitely many $n$.
By the restricted abc conjecture~\eqref{E:abc general position} applied to $P_n=(\phi x_n:1-\phi x_n:1)$,
\[
h(P_n)\leq\ld(P_n)+r_\T(P_n)+\psi_G(h(P_n))
\]
for all but finitely many $n$.
Let $D=\phi^*\T$.
By Lemma~\ref{L:C P},
there exists an~$M$ such that
\[
h(P_n)\leq\ld( x_n)+r_D( x_n)+M+\psi_G(h(P_n)).
\]
Let $D'$ be the reduced divisor of $D$.
By~\eqref{E:r D'},
$r_D(x_n)\leq \deg D'+h_{D'}(x_n)$.
Further,
$h_{D'}(x_n)\leq h(P_n)-h_\omega(x_n)$
by~\eqref{E:omega} and~\eqref{E:h phi}.
We find
\[
h_\omega(x_n)\leq\ld( x_n)+\deg D'+M+\psi_G(h(P_n)).
\]
Since by~\eqref{E:sequence},
$h_\omega(x_n)-\ld( x_n)>(\eps/2)h_\omega(x_n)$,
 this implies that
\[
\psi_G(h(P_n))>(\eps/2) h_\omega(x_n)-\deg D'-M.
\]
Since $h_\omega(x_n)\geq \frac12h(P_n)$ for Fermat curves of exponent at least $6$,
this contradicts that $\psi_G(h)=o(h)$.
We conclude that an infinite sequence such as~\eqref{E:sequence} does not exist.
\end{proof}

\begin{remark}\label{R:effective?}
The error term $\eps h+M$ in Theorem~\ref{T:M}
is ineffectively of order $o(h)$,
even if~\eqref{E:abc general position} would be known with an effective error term~$\psi_G(h)$.
The reason
is that different Fermat curves require different Belyi maps in the proof of Lemma~\ref{L:D=0},
and the determination of $M(\eps)$ becomes very complicated as~\mbox{$\eps\to0$} (and hence~$n\to\infty$).

By Section~\ref{S:uniform Belyi},
there is a finite number of Belyi maps that can be used,
the number depending on $\eps$ and $V$.
This could make our result effective.
One problem is that the transcendental limit points may be mapped arbitrarily close to $\T$.
\end{remark}

\begin{corollary}
Let\/ $d\geq1$ be a degree and assume~\eqref{E:abc general position}.
Then~\eqref{E:Mordell} holds
for all but finitely many points on the Fermat curve defined over any number field of degree at most\/ $d$ over\/ $\Q$.
\end{corollary}

\begin{proof}
By Remark~\ref{R:finite degree},
for each $v\in V$ there exists a number field $K$,
of degree bounded by a function of $d$,
 such that $K_v$ contains every number field of degree at most $d$ over $\Q$.
Thus for every $v\in V$,
there are finitely many embeddings to consider,
and the proof of Lemma~\ref{L:D=0} applies.
\end{proof}

\begin{theorem}\label{T:M}
Let\/ $K$ be a number field and\/ $\eps>0$.
Assume the restricted abc conjecture~\eqref{E:abc general position}.
Then there exists an\/ $M$ such that for every point\/ $P=(a:1-a:1)$ in\/ $K,$
\[
h(P)\leq\ld(P)+ r(P)+\eps h(P)+M.
\]
\end{theorem}

\begin{proof}
Cover $\P^1$ by the Belyi map~\eqref{E:phi Fermat}.
The point $P$ lifts to a point $x$,
defined over an extension of $K$ of degree at most $n^2$.
By the choice~\eqref{E:n} of $n$,
the degree of~$\phi$ is bounded by~$(2g-2)(1+\eps/2)$,
and by~\eqref{E:h phi} and~\eqref{E:omega},
\[
h(P)\leq h_\omega(x)+(\eps/2) h_\omega(x).
\]
We use Lemma~\ref{L:D=0} to obtain $h(P)\leq\ld(x)+\eps h_\omega(x)$.
To estimate the error term,
we use
$ h_\omega(x)\leq h(P)$.
It follows that
\[
h(P)\leq\ld(x)+\eps h(P).
\]
By Lemma~\ref{L:C P e},
we conclude that
$
h(P)\leq\ld(P)+r_\T(P)+\eps h(P)+M.
$
\end{proof}

Again,
the number field $K$ is inessential,
and we can prove the same theorem for all number fields of bounded degree.

\begin{corollary}
Let\/ $d\geq1$ be a degree and\/ $\eps>0$.
Assume~\eqref{E:abc general position} for all number fields of degree at most\/ $dn^2,$
where\/ $n\geq3+6/\eps$.
Then there exists an\/ $M$ such that
\[
h(P)\leq\ld(P)+ r(P)+\eps h(P)+M
\]
for every point\/ $P=(a:1-a:1)$ of degree at most\/ $d$ over\/ $\Q$.
\end{corollary}

\end{document}